\documentclass[10pt]{article}
\usepackage{amsmath,amsthm,amscd,amssymb,mathrsfs,setspace}
\usepackage{latexsym,epsf,epsfig}
\usepackage{epstopdf}
\usepackage{hyperref}
\usepackage[hmargin=3.5cm,vmargin=3.5cm]{geometry}
\usepackage{color}
\newcommand{\ds}{\displaystyle}

\newcommand{\reals}{\mathbb{R}}
\newcommand{\realstwo}{\mathbb{R}^2}
\newcommand{\realsthree}{\mathbb{R}^3}
\newcommand{\xb}{{\bf{x}}}

\newcommand{\Dx}{{\partial_x}}
\newcommand{\pd}{{\partial}}

\newcommand{\cT}{{\mathscr{T}}}
\newcommand{\Dn}{\partial_{\nu}}
\newcommand{\Dz}{{\partial_z}}

\newcommand{\cF}{\mathscr{F}}

\newcommand{\R}{\mathbb{R}}

\newcommand{\Om}{{\Omega}}

\newcommand{\bH}{\mathbf{H}}

\newcommand{\cH}{\mathcal{H}}

\theoremstyle{plain}
\newtheorem{theorem}{Theorem}[section]
\newtheorem{condition}{Condition}
\newtheorem{lemma}[theorem]{Lemma}

\newtheorem{corollary}[theorem]{Corollary}

\theoremstyle{remark}
\newtheorem{remark}{Remark}[section]
\numberwithin{equation}{section}
\numberwithin{theorem}{section}
\numberwithin{remark}{section}

\title{Von Karman plate in a gas flow:  recent results and conjectures}
\date{\today}

 \author{\normalsize \begin{tabular}[t]{c@{\extracolsep{.6em}}c@{\extracolsep{.6em}}c@{\extracolsep{.6em}}c}
         Igor Chueshov  & Earl H. Dowell  & Irena Lasiecka & Justin T. Webster  \\
\it        Kharkov National Univ. & \it Duke University & \it Univ. of Memphis    &\it  College of Charleston \\
\it        Kharkov, Ukraine & \it Durham, NC & \it Memphis, TN &\it Charleston, SC\\
\it       chueshov@karazin.ua & dowell@duke.edu&  \it lasiecka@memphis.edu&  \it websterj@cofc.edu
\end{tabular}}

\begin{document}
\maketitle

\begin{abstract} {\noindent
We give a survey of recent results on
flow-structure interactions  modeled  by a modified wave equation coupled at an interface with equations of nonlinear elasticity. 
Both subsonic and supersonic flow velocities are considered.  The focus of the discussion here is on the interesting {\em mathematical aspects} of physical phenomena occurring in aeroelasticity, such as flutter 
and divergence. This leads to a partial differential equation (PDE) treatment of issues such as well-posedness of finite energy solutions, and long-time (asymptotic) behavior. The latter includes theory of asymptotic stability, convergence to equilibria, and to  global attracting sets. 
We complete the discussion with  several well known observations and conjectures based on experimental/numerical studies.\\
\noindent {\bf Key terms}: flow-structure interaction, nonlinear plate, flutter,
 well-posedness, long-time dynamics, global attractors. }
\end{abstract}

\section{Introduction}

The main goal of this paper is to provide
a  mathematical survey of various PDE models for fluttering plates \cite{bal0,bolotin,dowellnon,dowell1,gibbs,tang}. The {\em flutter pheonmenon} is of great interest across many fields. Flutter is a sustained, systemic instability which occurs as a feedback coupling between a thin structure and an inviscid fluid.
%
%
%
%
A static bifurcation may also occur, known as {\em divergence} or buckling. 
%
Here we consider flow-plate dynamics corresponding to large range of  flow velocities  including  both  {\it subsonic and supersonic.}

Flow-structure models have attracted considerable attention in the past mathematical literature, see, e.g., \cite{bal4,bal3,b-c, LBC96,b-c-1,chuey,springer,jadea12,dowellnon,webster}
and the references therein.
However, the majority of the work that
has been done on flow-structure interactions has been devoted to numerical and experimental studies, see,
for instance,
\cite{bal0,bolotin,dowell1,dowellnon,HP02} and also the survey~\cite{Li03}
and the literature cited there.  Many mathematical studies have been based on linear, two dimensional, plate models  with specific geometries, where  the primary goal was to determine the speed at which flutter occurs
\cite{bal0,bolotin,dowell1,HP02,Li03}, see also \cite{bal4,shubov2,bal-shubov,bal-shubov1,shubov1} for the recent 
studies  of linear models with a one dimensional flag-type structure.
%
%

Given the difficulty of modeling coupled PDEs at an interface \cite{cbms,redbook}, theoretical results have been sparse.
%
%
While numerical studies address various situations, they are based on finite dimensional approximations of continuum models fully described by PDEs \cite{bal0,bolotin,b-c}. The infinite dimensional nature of the physical phenomena may not be adequately reflected by these approximations. 
%
%
The results presented herein have demonstrated that the models can be studied from an infinite dimensional point of view, and moreover that meaningful statements can be made about the physical mechanisms in flow-structure interactions {\em strictly from the PDE model}.  
The problems in the PDE analysis revolve around (i) the mismatch of regularity between two types of dynamics: the flow and the structure  which are coupled in a hybrid way,  
 (ii) the physically required presence of unbounded or ill-defined terms (boundary traces in the coupling conditions), and (iii) intrinsically non-dissipative generators 
of dynamics, even in the linear case. (The latter are associated with potential chaotic behavior.) 


 This paper specifically addresses the interactive dynamics between a nonlinear plate and a surrounding potential flow \cite{bolotin,dowellnon}.
%
%
  The class of models  under consideration  is standard in the modeling of flow-structure interactions, and goes back to classical discussions  (\cite{bolotin,dowellnon}, and also \cite{B,dowell1} and the references therein).
The description of  physical phenomena such as flutter, divergence or buckling  will translate into mathematical questions  related to 
 existence of  nonlinear semigroups representing  a  given dynamical system,  asymptotic stability of trajectories, and convergence  to equilibria or to compact attracting sets. 
   Interestingly enough, different model descriptions (including boundary conditions)  lead to an array of  diverse mathematical issues that involve  not only classical  PDEs, but  subtle questions in 
  non-smooth elliptic theory, harmonic analysis, and singular operator theory.  
Thus, in this paper, we concentrate primarily on mathematically rigorous aspects of modeling aeroelastic dynamics and their analysis. For more details concerning  the relationship of the theory developed with numerics and experiments we refer to the paper \cite{CDLW-eng}.
  
%

\medskip\par\noindent
{\bf  Notation:}
 For the remainder of the text denote $\xb$ for $(x,y,z) \in \realsthree_+$ or $(x,y) \in \Omega \subset \partial\realsthree_+$, as dictated by context. Norms $||\cdot||$ are taken to be $L_2(D)$ for the domain $D$. The symbol $\nu$ will be used to denote the unit normal vector to a given domain, again, dictated by context. Inner products in $L_2(\realsthree_+)$ are written $(\cdot,\cdot)$, while inner products in $L_2(\R^2\equiv\pd\R^3_+)$ are written $\langle\cdot,\cdot\rangle$. Also, $ H^s(D)$ will denote the Sobolev space of order $s$, defined on a domain $D$, and $H^s_0(D)$ denotes the closure of $C_0^{\infty}(D)$ in the $H^s(D)$ norm denoted by $\|\cdot\|_{H^s(D)}$ or $\|\cdot\|_{s,D}$.  We make use of the standard notation for the boundary trace of functions, e.g., for $\phi \in H^1(\realsthree_+)$, $ \gamma (\phi) \equiv tr[\phi]=\phi \big|_{z=0}$ is the trace of $\phi$ on the plane $\{\xb:z=0\}$.  

\section{PDE Description of the Gas-Plate Model}
%
%
 To describe the behavior of the gas (inviscid fluid), the theory of potential flows \cite{bolotin,dowellnon,dowell1} is utilized. The dynamics of the plate are governed by plate equations with the  von Karman (vK) nonlinearity \cite{springer,ciarlet}.
  This model is appropriate for plate dynamics with ``large" displacements,  and therefore applicable for the flexible structures of interest \cite{bolotin,ciarlet,dowellnon,dowell1,lagnese}.

The environment considered is $\realsthree_+ = \{(x,y,z) \in\realsthree\, :\, z\ge 0\}$. The thin plate has negligible thickness in the $z$-direction. The unperturbed flow field moves in the $x$-direction at the fixed velocity $U$. The physical constants have been normalized (corresponding to the linear potential flow) such that $U=1$ corresponds to Mach 1.  The plate
 is modeled by  a bounded domain $\Omega \subset \partial \realsthree_+$ with smooth boundary $\partial \Omega = \Gamma$. 
  
 The scalar function $u: \Om\times \R_+ \to \reals$ represents the displacement of the plate in the $z$-direction at the point $(x,y)$ at the moment $t$. The plate is taken with general boundary conditions $BC(u)$ which will be specified later. The dynamics are thus given by:
\begin{equation}\label{plate*}\begin{cases}
       (1-\alpha\Delta)u_{tt}+\Delta_{x,y} ^2 u+D_0(u_t)+f(u)=p(\xb,t) ~\text{in}~\Omega\\
       BC(u) ~ \text{on} ~\partial \Omega,
  \end{cases} \end{equation}
with appropriate initial data $u_0,~u_1$. 
The quantity $p(\xb,t)$ corresponds to the aerodynamic pressure of the flow on the plate and, in this configuration, is given in terms of the flow: \begin{equation}\label{coupling}\ds p(\xb,t)=p_0(\xb)+(\partial_t +U\partial_x)tr[\phi],~~\xb\in \Omega.\end{equation}   The quantity $p_0$ represents static pressure applied on the surface of the plate. The function $D_0$ represents an abstract form of interior damping. A parameter
$\alpha \ge 0$ represents rotational inertia in the filaments of the plate and is proportional to the square of the thickness of the plate. The case of interest, $\alpha =0$, is said to be {\em non-rotational}; the mathematical effects of $\alpha$ are discussed below.
The vK nonlinearity is given by:  \newline $$\ds f(u)=-[u, v(u)+F_0],
$$ where $F_0$ is a given {\em in-plane load} of sufficient regularity. 
      The vK bracket  is given by
$$
\ds
[u,w] = u_{xx}w_{yy}+
u_{yy}w_{xx} -
2 u_{xy}w_{xy},
$$ 
and
the Airy stress function $v(u_1,u_2) $ solves the  elliptic
problem
\begin{equation*}
\Delta^2 v(u_1,u_2)+[u_1,u_2] =0 ~~{\rm in}~~  \Omega,\quad \Dn v(u_1,u_2) = v(u_1,u_2) =0 ~~{\rm on}~~  \Gamma,
\end{equation*}
 (the notation $v(u)=v(u,u)$ is employed).
 Here and in In what follows $\nu$ (resp., $\tau$) denotes the unit outer normal (resp., tangential)  to $\Omega $.
 \par 
 The following types of boundary conditions for the displacement $u$ are typical for flow-structure interaction models.
 \begin{itemize}
\item[(C)] {\it  Clamped boundary conditions} (corresponding to a panel element in the flow-structure interaction) take the form 
\begin{equation*} u = \Dn u = 0 ~\text{ on } \Gamma. \end{equation*} 
In this case, the plate is  considered to be embedded  in or affixed to  a large rigid body.
\item[(FC)]  Let the boundary be partitioned in two pieces: $\Gamma_0$ and $\Gamma_1$. {\it  Free-clamped boundary conditions} (possibly corresponding to a flap, flag, or cantilevered airfoil) are given by
 \begin{equation*}
\begin{cases} u=\partial_{\nu}=0~~\mathrm{on}~~ \Gamma_0\\ \Delta u + (1-\mu) B_1 u = D_1(\Dn u_t)~~\mathrm{on}~~ \Gamma_1,
 \\
\partial_{\nu} \Delta u + (1-\mu) B_2 u - \mu_1 u = D_2(u,u_t)  + \alpha \Dn u_{tt} ~~\mathrm{on}~~ \Gamma_1, \end{cases} \end{equation*}
 where the boundary operators $B_1$ and $B_2$
are given by \cite{lagnese}:
\begin{equation}
\begin{array}{c}
B_1u = 2 \nu_1\nu_2 u_{xy} - \nu_1^2 u_{yy} - \nu_2^2 u_{xx}=-\partial_{\tau}^2u-\nabla \cdot \nu(\xb)\Dn u\;, \\
\\ 
B_2u = \partial_{\tau} \left[ \left( \nu_1^2 -\nu_2^2\right) u_{xy} + \nu_1
\nu_2 \left( u_{yy} - u_{xx}\right)\right]\,=\partial_{\tau}\partial_{\nu}\partial_{\tau}u;%
\end{array}
\end{equation}
  The parameter $\mu_1$ is nonnegative and $0<\mu<1$ is the Poisson modulus. The operators $D_i$ for $i=1,2$ are possible energy damping/dissipation functions. 
\end{itemize}
 The scalar function $\phi: \realsthree_+\times \reals \to \reals$ is the linear {\em flow potential} and satisfies:
\begin{equation} \label{flow}\begin{cases}
     (\partial_t+U\partial_x)^2\phi=\Delta_{x,y,z} \phi ~\text{in} ~\realsthree_+\\
    FC(\phi),
        \end{cases} \end{equation}
       with appropriate initial data $\phi_0$ and $\phi_1$\footnote{Note that, as mentioned above, $U$ is normalized by the speed of sound and thus is what engineers refer to as the {\em Mach number.}}. The term $FC(\phi)$ represents the {\em flow boundary conditions} or {\em interface conditions}. We will consider two primary flow conditions:
        \begin{itemize}
        \item[(NC)] The standard flow boundary conditions---the {\em full Neumann condition}---are henceforth denoted (NC). This is typically utilized when a majority of the plate boundary is clamped or hinged. It takes the form:
        \begin{equation}\label{NC}
        \partial_z\phi \big|_{z=0} = (\partial_t+U\partial_x)u_{\text{ext}},~~\text{on }~ \mathbb{R}^2=\partial \mathbb{R}^3_+,
        \end{equation}
         where the subscript `ext' indicates the extension by zero for functions defined on $\Omega$ to all of $\mathbb R^2$. 
        \item[(KJC)] The {\em Kutta-Joukowsky} flow condition, henceforth denoted (KJC), is a mixed (Zaremba-type \cite{savare}) condition which will be discussed at length below.  We introduce the {\em acceleration potential} $\psi = (\partial_t+U\partial_x)\phi$ for the flow potential $\phi$. Then the  {\it Kutta-Joukowsky} flow condition on $\realstwo$ is
        \begin{equation}\label{KJC}
        \partial_z \phi = (\partial_t+U\partial_x)u ~ \text{ on } ~\Omega;~~ \psi = 0  \text{ on } ~ \realstwo \setminus \Omega.
        \end{equation}
        \end{itemize}
   In both cases above the interface coupling on the surface of the plate occurs in a Neumann type boundary condition known as the {\em downwash} of the flow. 
   The choice of flow conditions itself dependent upon the plate boundary conditions, and this is determined by application. 
 \section{Standard Panel  Model}
 Our primary interest here is the following PDE system\footnote{For definiteness and some simplification we concentrate on the clamped boundary conditions 
for the displacement $u$.}
\begin{equation}\label{flowplate}\begin{cases}
(1-\alpha\Delta)u_{tt}+\Delta^2u+ku_t +f(u)= p_0+\big(\partial_t+U\partial_x\big)tr[\phi] & \text { in } \Omega\times (0,T),\\
u(0)=u_0;~~u_t(0)=u_1,\\
u=\Dn u = 0 & \text{ on } \partial\Omega\times (0,T),\\
(\partial_t+U\partial_x)^2\phi=\Delta \phi & \text { in } \realsthree_+ \times (0,T),\\
\phi(0)=\phi_0;~~\phi_t(0)=\phi_1 & \text { in } \realsthree_+\\
\Dn \phi = -\big[(\partial_t+U\partial_x)u (\xb)\big]_{\text{ext}}& \text{ on } \realstwo_{\{(x,y)\}} \times (0,T).
\end{cases}
\end{equation}
 In the discussion below, we will encounter strong (classical), generalized (mild), and weak (variational) solutions.
In the results below, semigroup theory is utilized, which requires the use of \textit{generalized} solutions; these are strong limits of strong solutions. These solutions satisfy an integral formulation of (\ref{flowplate}), and are called \textit{mild} by some authors. We note that generalized solutions are also weak solutions, see, e.g., \cite[Section 6.5.5]{springer} and \cite{webster}. For the exact definitions of 
 strong and generalized  solutions
we refer  to \cite{springer,jadea12,supersonic,webster}.  Finite energy solutions are identified with mild solutions.

\subsection{Well-Posedness of Nonlinear Panel Model}\label{wellp}

Well-posedness results in the past literature deal mainly with the dynamics possessing some regularizing effects. This has been accomplished by either accounting for non-negligible rotational inertia \cite{b-c,b-c-1,springer} and strong damping of the form $-\alpha \Delta u_t$, or by incorporating helpful thermal effects  into the structural model \cite{ryz,ryz2}.  In the cases listed above, the natural structure of the dynamics dictate that the plate velocity has the property $u_t \in H^1(\Omega)$,
 which provides  the needed regularity 
for the applicability of many standard tools in nonlinear analysis. One is still faced with the low regularity of boundary traces,
due to the failure of the Lopatinski conditions \cite{sakamoto}.
In fact, the first contribution to the mathematical analysis of the problem  is \cite{LBC96,b-c-1} (see also \cite[Section 6.6]{springer}), where the case  $\alpha > 0$ (rotational) is fully treated.  The method employed in \cite{LBC96,b-c-1,springer}
relies on sharp microlocal estimates for the solution to the wave equation driven by $H^{1/2}(\Omega) $ Neumann  boundary data given by~ $u_t + U u_x$. This gives~ $\phi_t|_{\Omega} \in
L_2(0,T; H^{-1/2}(\Omega))$  \cite{miyatake}.  Along with an explicit solver for the three dimensional flow equation (the Kirchhoff formula), and a Galerkin approximation for the structure, one may construct  a fixed point for the appropriate
 solution map. Such a method  works well for all values of $U$. 
%
%
When $\alpha =0$, and additional smoothing is not present, the corresponding estimates become singular, destroying the  applicability of the  previous methods.  The  main  mathematical difficulty of this problem is  the presence of  the   boundary  terms:  $ (\phi_t +
 U \phi_x)|_{\Omega} $  ~acting as the aerodynamic pressure on the plate. When $U =0$, the corresponding  boundary terms exhibit monotone behavior with respect to   the natural energy  inner product (see \cite[Section 6.2]{springer} and \cite{cbms}) which is topologically equivalent  to  the topology of the state space $Y$.   The latter  enables  the use of monotone operator theory \cite{springer} (and references therein).  However, when
  $U > 0$ this is no longer true  with respect to the  topology induced by the energy spaces.  
%

We begin with  an overview of the well-posedness results for the flow-plate model which covers both subsonic and supersonic cases.

	\begin{theorem}\label{th:supersonic}
Consider  problem in (\ref{flowplate}) with $U\ne1$, $p \in L_2(\Omega), F_0 \in H^{3 +\delta}(\Omega)$.  Let $ T > 0 $ and
\begin{equation}\label{space-Y}
( \phi_0, \phi_1; u_0, u_1 ) \in  
Y = Y_{fl}\times Y_{pl} \equiv \big(H^1(\realsthree_+) \times L_2(\realsthree_+)\big)\times\big( H^2_0(\Omega) \times L^{\alpha}_2(\Omega)\big)
\end{equation}
 where 
 $L_2^{\alpha} (\Omega) =  
 H^1_0(\Omega)$ if $\alpha>0$ and  $L_2^{\alpha} (\Omega)=L_2(\Omega)$ when $\alpha =0$.
Then there exists unique generalized solution
\begin{equation}\label{phi-u0reg}
(\phi (t), \phi_t(t); u(t), u_t(t)) \in C([0, T ],  Y).
\end{equation}
This solution is also weak and generates
  a  nonlinear  continuous  semigroup
$S_t : Y \rightarrow Y$.
\par
 For any initial data in
  \begin{equation}\label{Y1}
Y_1 \equiv \left\{ y= (\phi,\phi_1;u,u_1)  \in Y\;   \left| \begin{array}{l}
\phi_1 \in H^1(\R^3_{+} ),~~ u_1 \in \cH,  ~\\ -U^2 \Dx^2 \phi  + \Delta \phi \in L_2(\R^3_{+}), \\
\Dn  \phi = - [u_1 +U \Dx u ]\cdot {\bf 1}_{\Omega} \in H^1(\R^2), \\
~
-\Delta^2 u + Utr[ \Dx \phi] \in [L_2^\alpha(\Omega)]^* \end{array} \right. \right\}
\end{equation}
   the corresponding  solution is also strong. Here
   $[L_2^\alpha(\Omega)]^*$ is  
$ H^{-1}(\Omega)$ if $\alpha>0$ and  $L_2(\Omega)$ when $\alpha =0$.

\end{theorem}
	
	In the subsonic case, there is no spatial ``degeneracy" in the flow equation, leading to a stronger well-posedness result:
\begin{theorem}\label{th:subsonic} Suppose $U<1$, $p \in L_2(\Omega)$ and $F_0 \in H^{3+\delta}(\Omega)$.

\item \textbf{Generalized and Weak Solutions:} Assume 
that initial data sytisfy \eqref{space-Y}. Then \eqref{flowplate} has a unique generalized solution on any interval $[0,T]$. 
 Every generalized solution is also weak.
Moreover these solutions to \eqref{flowplate} generates
a  dynamical system $(S_t,Y)$. 
Any generalized (and hence weak) solution to \eqref{flowplate},  satisfies the bound \begin{equation}\label{energybound}
\sup_{t \ge 0} \left\{((1-\alpha \Delta)u_t, u_t)_{0,\Omega}^2+\|\Delta u\|_{\Omega}^2+\|\phi_t\|_{\realsthree_+}^2+\|\nabla \phi\|_{\realsthree_+}^2 \right\}  < + \infty.\end{equation}

 \noindent \textbf{Strong Solutions:} Assume $u_1 \in H^2_0(\Omega), ~u_0\in W \equiv\{w\in H_0^2(\Omega) : \Delta ^2 w \in L_2^\alpha(\Omega)\}.$ Also, denote $\mathcal H \equiv H_0^2(\Omega) \times L_2^{\alpha}(\Omega)$. Moroever, suppose $\phi_0 \in H^2(\realsthree_+)~\text{ and }~\phi_1 \in H^1(\realsthree_+)$, with the following compatibility condition in place: \begin{equation}\label{compatibilitycondition}
\partial_z \phi_0\big|_{z=0} = \begin{cases}u_1+U\partial_x u_0 & \text{ if } \xb \in \Omega\\ 0 & \text{ if } \xb \notin \Omega \end{cases}.
\end{equation} 
Then \eqref{flowplate} has a unique strong solution for any interval $[0,T]$. This solution possesses the properties \begin{equation*} (\phi;\phi_t;\phi_{tt}) \in L_{\infty}(0,T;H^2(\realsthree_+)\times H^1(\realsthree_+)\times L_2(\realsthree_+)),
~(u;u_t;u_{tt})\in L_{\infty}(0,T;W \times \mathcal H),
\end{equation*} and satisfies the energy identity 
\begin{equation}\label{ener-rel-subs}
{\mathcal{E}}(t)+  \int_s^t \int_{\Omega} k u_t^2={\mathcal{E}}(s),~~ t\ge s,
\end{equation}
where $\mathcal E(t) = E_{pl}(t)+E_{fl}(t)+E_{int}(t)$,
with:
\begin{align}
E_{pl} = & ~\dfrac{1}{2}[((1-\alpha \Delta)u_t, u_t)_{0,\Omega}^2+ ||\Delta u ||_{0,\Omega}^2+\frac{1}{2}||\Delta v(u)||_{0,\Omega}^2]-\langle F_0,[u,u]\rangle_{\Omega}+\langle p,u\rangle_{\Omega} \\
E_{fl}=  &~\dfrac{1}{2}[||\phi_t||^2_{0,\realsthree_+}+||\nabla \phi||_{0,\realsthree_+}^2-U^2||\phi_x||^2_{0,\realsthree_+}],~~E_{int} =  ~2U\langle\phi,u_x\rangle_{\Omega}.
\end{align}
\end{theorem}


The proof of Theorem \ref{th:subsonic} given in 
\cite{springer} ans
\cite{webster} relies on three  principal ingredients: (1) renormalization of the state space which yields shifted dissipativity for the dynamics operator (which is nondissipative in the standard norm on the state space); (2) the sharp regularity of Airy's stress function, which converts a supercritical nonlinearity into a critical one  \cite{springer} (and references therein): 
\begin{equation}\label{airy}
||v(u)||_{W^{2,\infty}(\Omega) } \leq C ||u||^2_{H^2(\Omega)};
\end{equation}
and (3) control of low frequencies for the system by the nonlinear source, represented by the inequality
\begin{equation}
||u|^2_{L_2(\Omega)} \leq \epsilon [ ||u||^2_{H^2(\Omega)^2} + ||\Delta v(u)||^2_{L_2(\Omega) } ] + C_{\epsilon} .
\end{equation}
An alternative proof based on a {\em viscosity method} can be found in \cite{jadea12}. 
\medskip\par

  In comparing the results obtained for supersonic and subsonic cases,  there are two major differences at the qualitative level, in addition to an obvious fact that $E_{fl}$ is no longer positive when $U > 1$:
{\it First}, the regularity of strong solutions obtained in the subsonic case \cite{jadea12,webster} coincides with regularity expected for classical solutions. In the supersonic case, there is a  loss of differentiability in the flow  in the tangential  $x$ direction, which then propagates to the  loss of differentiability in the structural variable $u$. As a consequence strong solutions do not exhibit sufficient regularity in order to perform the  needed calculations.
  To cope with the problem, special regularization procedure was introduced in \cite{supersonic},
  where  strong solutions are approximated by sufficiently regular functions, though not solutions to the given PDE.
  The limit passage allows one to obtain the needed estimates valid for the original solutions \cite{supersonic}.
The analysis of the supersonic case makes use of a change of state variable: the dynamic variable $\psi = \phi_t+U\phi_x$ is considered as the second flow state. In this case the adapted energies are:
\begin{align}
\widehat E_{fl}= ~ \dfrac{1}{2}[||\psi||^2_{0,\realsthree_+}+||\nabla \phi||_{0,\realsthree_+}^2],~~
\widehat {\mathcal E} =  ~E_{pl}+\widehat E_{fl},
\end{align}
and the energy identity is given by:
 \begin{equation} \label{energyrelationkj}
\widehat{\mathcal E}(t) + \int_s^t \int_{\Omega} k u_t^2 +\int_0^t\langle u_x,\psi\rangle_{\Omega} d\tau = \widehat {\mathcal E}(0).
\end{equation}

{\it Secondly}, in the subsonic case one shows that the solutions are {\it bounded} in time,
  see \cite[Proposition 6.5.7]{springer} and also \cite{jadea12,webster}.
  This property  cannot be shown in the supersonic analysis.  The leak of the energy in energy relation can not be compensated for by the nonlinear terms (unlike in the subsonic case). 

\subsection{Long-time Behavior and Attracting Sets}
 It is known \cite{dowell1} (and references therein) that, although the flow coupling may introduce potential instability (flutter), it can also help bound the plate dynamics and may assist in the dissipation of plate energy associated to high frequencies \cite{C,C1}.\footnote{Though, there are also calculations performed at the physical level which demonstrate that the effect of the flow can be destabilizing in certain regimes---see Remark \ref{piston} below.} The work in \cite{delay} is an attempt to address these observations rigorously, via the reduction result in Theorem \ref{rewrite} below. This result is implemented and the problem is cast within the realm of PDE with delay (and hence our state space for the plate contains a delay component which encapsulates the flow contribution to the dynamics). 
One would like to show that the model described above is stable in the sense that trajectories converge asymptotically to a ``nice" set. Unlike parabolic dynamics, there is no a priori reason to expect that hyperbolic-like dynamics can be asymptotically reduced to truly finite dimensional dynamics. By showing that the PDE dynamics converges to a finite dimensional (compact) \textit{attractor} it effectively allows the reduction of the analysis of the infinite dimensional, hyperbolic-like, unstable model (asymptotically in time) to a finite dimensional set upon which classical control theory can be applied.
\begin{remark}
Based on numerical results, it is well known that for a panel whose width is much larger than its chord, that four to six structural modes are sufficient to give very good accuracy to predict flutter and limit cycle oscillations. Indeed, for transonic and subsonic flows even one or two modes may be sufficient. As the width of the panel is decreased more and more modes are required however. Also if the plate is under tension, then the number of modes needed increases, and in the singular limit of infinite tension the number of modes required also goes to infinity \cite{B,dowellA}.\end{remark}

The main  point  of the treatment in \cite{delay} is to demonstrate the existence and finite-dimensiona\-lity of a global attractor for the reduced plate dynamics (see Theorem \ref{rewrite}) {\em in the absence of any imposed damping mechanism} and in the absence of any regularizing mechanism generated by the dynamics itself.
Imposing no mechanical damping, taking $\alpha =0$, and assuming the flow data are compactly supported yields the primary theorem in \cite{delay}:
\begin{theorem}\label{th:main2}
Suppose $0\le U \ne 1$, $F_0 \in H^{3+\delta}(\Omega)$ and $p_0 \in L_2(\Omega)$.
 Then there exists a compact set $\mathscr{U} \subset Y_{pl}$ of finite fractal dimension such that $$\displaystyle \lim_{t\to\infty} d_{Y_{pl}} \big( (u(t),u_t(t)),\mathscr U\big)
 \displaystyle=\lim_{t \to \infty}\inf_{(\nu_0,\nu_1) \in \mathscr U} \big( ||u(t)-\nu_0||_2^2+||u_t(t)-\nu_1||^2\big)=0$$
for any weak solution $(u,u_t;\phi,\phi_t)$ to \eqref{flowplate} with
initial data
$$
(u_0, u_1;\phi_0,\phi_1) \in Y,~~Y\equiv H_0^2(\Omega)\times L_2(\Omega)\times H^1(\realsthree_+)\times L_2(\realsthree_+),
$$
which are
localized  in $\R_+^3$ (i.e., $\phi_0(\xb)=\phi_1(\xb)=0$ for $|\xb|>R$ for some $R>0$). Additionally,there is the extra regularity $\mathscr{U} \subset \big(H^4(\Omega)\cap H_0^2(\Omega)\big) \times H^2(\Omega)$.
\end{theorem}

The proof of Theorem \ref{th:main2} requires modern tools and new long-time behavior technologies applied within this delay framework. Specifically, the approach mentioned above (and utilized in \cite{springer,ryz,ryz2}) {\em does not apply} in this case.  A relatively new technique \cite{chlJDE04,ch-l,springer} allows one to address the asymptotic compactness property for the associated dynamical system without making reference to any gradient structure of the dynamics (not available in this model, owing to the dispersive flow term). In addition, extra regularity and finite dimensionality of the attractor is demonstrated via a quasi-stability approach  \cite{springer}. The criticality of the nonlinearity and the lack of gradient structure prevents one from using a powerful technique of {\it backward smoothness of trajectories}  \cite{springer}, where smoothness is propagated backward from the equilibria.  Without a gradient structure, the attractor may have complicated structure (not being characterized by the equilibria points). In order to cope with this issue, a novel method that is based on suitable  construction of an $\epsilon$ net   and exploits only the compactness of the attractor.

In the presence of additional  damping imposed on the structure, it is reasonable to expect that the entire evolution (both plate and flow) is strongly stable in the sense of strong convergence to  equilibrium states. Such results are shown for the model  above (in the same references) only for the case of {\em subsonic flows}, $U<1$. The key to obtaining such results is a {\em viable energy relation} and a priori bounds on solutions which yield {\em finiteness of the dissipation integral}.
We obtain a result for stabilization of the full flow-plate dynamics (as opposed to just the structural dynamics, as above, for all flow velocities $U \ne 1$). 
For this we recall
  the global bound in \eqref{energybound}
which along with the energy identity in \eqref{ener-rel-subs} 
 allows us to obtain the following: \begin{corollary}\label{dissint}
Let $ k> 0 $ and $\alpha=0$. Then the dissipation integral is finite. Namely, for a generalized solution to \eqref{flowplate} we have $$  \int_0^{\infty} \|u_t(t)\|_{0,\Omega}^2 dt \leq  
K_{u,k} < \infty.$$
\end{corollary}
\noindent From this finiteness, we can show \cite{conequil1}:
\begin{theorem}\label{regresult}
Let $0\le U<1$ and $\alpha=0$.  Assume $p_0 \in L_2(\Omega)$ and $F_0 \in H^{4}(\Omega)$. Then for all $ k>0$, any  solution $(u(t),u_t(t);\phi(t),\phi_t(t))$ to the flow-plate system \eqref{flowplate} with 
initial data
$$
(u_0,u_1;\phi_0,\phi_1) \in (H_0^2\cap H^4)(\Omega) \times H_0^2(\Omega)\times H^2(\mathbb R^3_+) \times H^1(\mathbb R^3_+).
$$
that are
 spatially localized in the flow component (i.e., there exists a $\rho_0>0$ so that for $|\xb|\ge \rho_0$ we have $\phi_0(\xb) = \phi_1(\xb)=0$) has the property that 
 \begin{align*}\lim_{t \to \infty} \inf_{(\hat u,\hat \phi) \in \mathcal N}\left\{\|u(t)-\hat u\|^2_{H^2(\Omega)}+\|u_t(t)\|^2_{L_2(\Omega)}+\|\phi(t)-\hat \phi\|_{H^1( K_{\rho} )}^2+\|\phi_t(t)\|^2_{L_2( K_{\rho} )} \right\}=0
 \end{align*} 
 for any   $\rho>0$, where 
 $K_\rho=\{ \xb\in\R_+^3\, : \|\xb\|\le \rho\}$ and
 $\mathcal N$ denotes the set of stationary solutions to \eqref{flowplate} (for their 
 existence and properies see \cite{springer}).
\end{theorem}
The above result remains true for finite energy initial data (only in $Y$) with convergence in a {\em weak sense}. 
Additionally, for the system with rotational inertia present $\alpha>0$ (and corresponding strong damping) the analogous {\em strong} stabilization result holds \cite{chuey,springer} for finite energy initial data 
$$
(u_0,u_1;\phi_0,\phi_1) \in H_0^2(\Omega) \times L^{\alpha}_2(\Omega)\times W_1(\mathbb R^3_+) \times L_2(\mathbb R^3_+).
$$

The fact that $ k>0 $ is needed in order to obtain convergence to equilibria is corroborated by the counterexample  \cite{igor}, 
 which shows that  with $ k =0$ periodic solutions may remain in the limiting dynamics. In fact, $ k >0 $ allows one to prove that the entire flow-structure system is a gradient system. In view of the above result---Theorem \ref{regresult}---with finite energy initial data, we see that {\em any damping} imposed on the structure seems to eliminate the flutter. This is consistent with experiment for clamped plates. From a physical point of view, the aerodynamic damping in subsonic flow for a plate is much lower than in supersonic flow, so the plate tends to oscillate  in  a neutrally stable state in the absence of aerodynamic damping or structural damping \cite{dowellA}. On the other hand, when $ U > 1 $ we expect that such result no longer holds (see, e.g., the numerical analysis   in \cite{jfs}).

We also note that
when the vK nonlinearity is replaced by Berger's nonlinearity, and the damping coefficient $ k > 0$ is taken sufficiently large, it has been shown \cite{conequil2} that 
all {\em finite energy} trajectories converge strongly to equilibria, and hence the flutter is eliminated. Though we have Theorem \ref{regresult} for {\em smooth initial data}, the analogous result for the vK nonlinearity   is  still unknown, unless additional static damping controller is used \cite{conequil1}.

\subsection{Reduced Delayed Model}
A key to obtaining attracting sets is the representation of the flow on the structure via a delay potential (see Section 3.3 in \cite{springer}).  Reducing this full flow-plate problem to a delayed von Karman plate is the primary motivation for our main result and permits a starting-point for long-time behavior analysis of the flow-plate system, which is considerably more difficult otherwise. The exact statement of this key reduction is given in the following assertion:

\begin{theorem}\label{rewrite}
Let the hypotheses of Theorem~\ref{th:main2} be in force, and $(u_0,u_1;\phi_0,\phi_1) \in \cH \times L_2(\Omega) \times H^1(\realsthree_+) \times L_2(\realsthree_+)$. Assume that there exists an $R$ such that $\phi_0(\xb) = \phi_1(\xb)=0$ for $|\xb|>R$.  Then the there exists a time $t^{\#}(R,U,\Omega) > 0$ such that for all $t>t^{\#}$ the weak solution $u(t)$ to (\ref{flowplate}) satisfies the following equation:
\begin{equation}\label{reducedplate}
u_{tt}+\Delta^2u-[u,v(u)+F_0]=p_0-(\partial_t+U\partial_x)u-q^u(t)
\end{equation}
with
\begin{equation}\label{potential}
q^u(t)=\dfrac{1}{2\pi}\int_0^{t^*}ds\int_0^{2\pi}d\theta [M^2_{\theta} u_{\text{ext}}](x-(U+\sin \theta)s,y-s\cos \theta, t-s).
\end{equation}
Here, $M_{\theta} = \sin\theta\partial_x+\cos \theta \partial_y$ and \begin{equation}\label{delay} t^*=\inf \{ t~:~\xb(U,\theta, s) \notin \Omega \text{ for all } \xb \in \Omega, ~\theta \in [0,2\pi], \text{ and } s>t\}
\end{equation} with $\xb(U,\theta,s) = (x-(U+\sin \theta)s,y-s\cos\theta) \subset \realstwo$.
\end{theorem}


The flow state variables $(\phi,\phi_t)$  manifest themselves in our rewritten system via the delayed  character of the problem; they appear  in the initial data for the delayed  component of the plate, namely $u^t\big|_{(-t^*,0)}$. 
 The reduced model displays the following features: (i) it does not have gradient structure (due to dispersive and delay terms), (ii) the delay term appears at the critical level of regularity. This raises a natural question regarding long time behavior and the existence of global attractors. 
 However,  despite of the lack of gradient structure, compensated compactness methods allow to ``harvest" some compactness properties from the reduction, so the ultimate dynamics does admit global attracting set. 
 In fact, the result of Theorem \ref{th:main2} is a consequence of a more general result pertinent to delay PDE system, and formulated in Theorem \ref{co:generation}. 

 \section{Simplified/Reduced Models}
\subsection{Plate Model with Delay as a Dynamical System   }
Below we utilize a positive parameter $0<t^*<+\infty$ as the time of delay, and
accept  the commonly  used (see, e.g.,   \cite{Delay-book1995} or \cite{wu-1996}) notation $u^t(\cdot)$
 for function on $s\in [-t^*,0]$ of the form $s\mapsto u(t+s)$.
This is necessary due to the delayed character of the problem
which requires initial data of the prehistory interval  $[-t^*,0]$, i.e., we
need to impose an initial condition
 of the form $u|_{t \in (-t^*,0)} = \eta(\xb, t)$,
 where $\eta$ is a given function on $\Om\times [-t^*,0]$.
 We can choose this prehistory data $\eta$ in various
 classes. In our problem it is convenient to deal
 with Hilbert type structures, and therefore we assume in the further considerations that
 $ \eta \in L_2(-t^*,0;\cH)$. Since we do not assume
 the continuity of $\eta$ in $s\in [-t^*,0]$,  we also
 need to add the (standard) initial conditions of the form
 $u(t=0)= u_0(\xb)$ and $\partial_t u(t=0)=u_1(x)$.
 \par
The delayed system (without rotational inertia) is then given by:
\begin{equation}\label{plate}\begin{cases}
u_{tt}+\Delta^2u+k_0u_t+f(u) +Lu= p_0+q(u^t,t) ~~ \text { in } ~\Omega\times (0,T), \\
u=\Dn u =0  ~~\text{ on } ~ \partial\Omega\times (0,T),  \\
u(0)=u_0,~~u_t(0)=u_1,~~\\ u|_{t \in (-t^*,0)} = \eta\in L_2(-t^*,0;H^2_0(\Omega)).
\end{cases}
\end{equation}
Here the forcing term $q(u^t,t)$ occurring on the RHS of the plate equation encompasses the gas flow. The operator $L$ encompasses any spatial lower order terms which do not have gradient structure.

 We take a \textit{weak solution} to \eqref{plate} on $[0,T]$ to be a function $$u \in L_{\infty}(0,T;H_0^2(\Omega))\cap W^1_{\infty}(0,T;L_2(\Omega)) \cap L_2(-t^*,0;H_0^2(\Omega))$$ such that the variational relation corresponding to \eqref{plate} holds (see, e.g., \cite[(4.1.39), p.211]{springer}).
In order to consider the delayed system as a dynamical system with the phase space $$\mathbf{H}\equiv \cH\times L_2(\Omega)\times L_2(-t^*,0;\cH).
 $$ 
 With the above notation we introduce the  operator
$ S_t\, : \bH\mapsto \bH$ by the formula
\begin{equation}\label{semigroup}
  S_t(u_0, u_1, \eta) \equiv (u(t), u_t(t), u^t),
\end{equation}
where $u(t)$ solves \eqref{plate} 
and $u^t(s) \equiv u(t+s)$, $s \in [-t^*, 0 ]$.

 \begin{theorem}\label{co:generation}
 Under the hypotheses imposed,  \eqref{plate} generates
   a strongly continuous semigroup $S_t : \mathbf{H} \rightarrow \mathbf{H}$ which possesses a compact global  and "smooth" attractor of finite fractal dimension.\footnote{
For the corresponding definitions and basic facts we refer
to \cite{springer,temam}.
}
 \end{theorem}
 The main ingredient of the proof is a variant of {\it hidden compactness }
 of the term $q(u^t,t)$; indeed, a hidden regularity of the delay term leads to a variant of  compensated compactness property which allows one to derive {\it quasi stability} inequality for the dynamics (see \cite{delay}).   We note that quasi-stability method  originated in \cite{chlJDE04} (see also \cite{ch-l,springer,cl-hcdte},
 and is exposited in the recent monograph \cite{ch-dqsds}). For delay models
 this method was also applied recently in \cite{CR2014,ChuRez-ArX2014}.

Nonlinear PDEs with delay have been considered in various sources (see \cite{wu-1996} and references therein). In relation to plate equations with delayed  aerodynamic pressure, \cite{springer} (see also \cite{Chu92b}) provides a rather complete analysis of the delayed vK plate in the presence of rotational terms (and application to flow-plate interactions), and the references \cite{oldchueshov1,oldchueshov2} deal with the plate with delay in the presence of other nonlinear terms.
One should stress that even at the formal level---with or without rotational inertia---the reduced delay problem is challenging. This is due to the fact that the underlying delay system is intrinsically non-gradient---the structure of the aeroelastic potential provides non-conservative and non-dissipative terms that destroy the Lyapunov nature of the energy relation. In view of this, results on asymptotic behavior are challenging. 

\subsection{Piston Theory}
Classical piston theory (or law of plane sections \cite{bolotin,dowellnon,oldpiston}) replaces the {\em acceleration potential of the flow}, $\psi= tr[\phi_t+U\phi_x]$ with $-[u_t+Uu_x]$ in the RHS of the reduced plate equation \eqref{reducedplate}.

%
%
Formally, we can  arrive at the aforementioned, standard (linear) piston theory model by
 utilizing the reduction result in Theorem \ref{rewrite}; a bound  exists \cite[p.334]{springer}: \begin{equation}
||q^u(t)||_{0,\Omega}^2 \le \dfrac{C}{U} \int _{t-\frac{c}{U}}^t||\Delta u(\tau)||_{\Omega}^2 d\tau
\end{equation}
for $U$ sufficiently greater than 1. The constants $c,C>0$ depend only on the diameter of $\Omega$. This indicates that the delay term, $q^u$ decays in the appropriate space as $U$ increases. Hence, 
it is reasonable to guess that $q^u$ can be neglected in the case of large speeds $U$.
\footnote{See also Remark~\ref{piston} for other (physically motivated) expressions 
the ``piston'' term).} 
So we arrive at the following model
\begin{equation}\label{plate-stand}\begin{cases}
u_{tt}+\Delta^2u+f(u) = p_0-[u_t+Uu_x] ~~ \text { in } ~\Omega\times (0,T), \\
u=\Dn u =0  ~~\text{ on } ~ \partial\Omega\times (0,T).
\end{cases}
\end{equation}
  This model was intensively studied in 
the literature for different types of boundary conditions and resistance damping forces, see \cite{springer} and also \cite{ch-l} for related second order abstract models. The typical result is the following assertion:
\begin{theorem}
Under the conditions imposed on the model
the equations \eqref{plate-stand} generates a dynamical system in the space $H_0^2(\Omega)\times L_2(\Omega)$ possessing a compact global attractor of finite fractal dimension which is also ``smooth".	
\end{theorem}
A key avenue of research for supersonic flow-plate models in the standard panel configuration \cite{supersonic} (see Section \ref{wellp}) is an explicit comparison between solutions to the piston theoretic plate model and solutions to the clamped flow-plate interaction as $U$ becomes sufficiently large. 
{\em This comparison should be made rigorous}.  This leads to the following fundamental question from modeling point of view:
\begin{description}
\item {\bf Open Question:}
Show that the associated solutions to \eqref{reducedplate} and \eqref{plate-stand} discussed above coincide for (i) $t \to \infty$ with $U$ fixed and sufficiently large, or (ii) as $U \to \infty$ (the so called hypersonic limit) on some {\em arbitrary time interval $[T_1,T_2]$} (where $T_2$ is perhaps $\infty$). Such results would provide rigorous justification (at the infinite dimensional level) for piston theoretic approaches to flow-plate models.\footnote{
The reference \cite{hypersonic} contains a result for the behavior of solutions to the piston theoretic plate 
(with RHS given by $p^*$ above) as $U\to \infty$. The result, however, is only valid for arbitrarily small time intervals, and hence does not provide information about the behavior of solutions for arbitrary $t$. 
}
\end{description}

\begin{remark}\label{piston} In addition to the classical, linear piston theoretic model discussed here (in \eqref{plate-stand}), there are other {\em piston theories}.
For instance, if we a priori  assume that we deal with ``low'' frequencies regimes for the plate, then we can use the following expression for the aerodynamic pressure (see, e.g., \cite{B,dowellA} and the references therein):
\begin{equation*}
p(\xb,t) =p_0- \dfrac{U}{\sqrt{U^2-1}}\left(\left[\dfrac{U^2-2}{U^2-1}\right]u_t+Uu_x\right).
\end{equation*}
in the supersonic case ($U>1$).
We note that this term has a fundamentally different structure, and its contribution to the dynamics is markedly different than the RHS of \eqref{plate-stand}.
It is commonly recognized that the damping due to the aerodynamic flow is usually substantially larger in {\em magnitude} than the damping due to the plate structure. However as we can see  the aerodynamic damping can be negative as well as positive.
 For instance, in the case $1<U < \sqrt{2}$
we obtain negative damping \cite{B,dowellA}. This  has a definite
impact on the dynamics, producing instabilities in the system. 
\end{remark}

\section{Kutta-Joukowsky Boundary Conditions}

Arguably, the  Kutta-Joukowsky  boundary conditions for the flow (\ref{KJC})  are the {\em most important} when modeling an airfoil  immersed in a flow of gas \cite{bal,bal0}. 
Not surprisingly, these boundary conditions---{\em dynamic and mixed} in nature---are also  the most challenging  from mathematical   stand point. This was recently brought to evidence in a long treatise \cite{bal0}. 
Various aspects of the problem---in both subsonic and supersonic regime---have been  discussed  in \cite{bal0} within the context of  mostly one dimensional structures. 
The aim of this section is to revisit the problem by putting it within the framework of modern harmonic analysis. 
  The recently studies  \cite{fereisel,KJ} show how the flow condition (KJC) interacts with the clamped plate in subsonic flows {\em in order to develop a suitable abstract theory} for this particular  model. 
   Though the analysis is subsonic for (KJC), utilizing the flow energy from the supersonic panel (as above in Theorem \ref{th:supersonic}) is effective in the abstract setup of this problem.  
In fact,  even in the subsonic case, the analysis of semigroup generation proceeds  through the technicalities developed  earlier for supersonic case \cite{supersonic}. The key distinction from the analysis of the clamped  flow-plate interaction (owing to the dynamic nature of the boundary conditions) is that a {\em dynamic flow-to-Neumann map} must be utilized to incorporate the boundary conditions into the abstract setup. The regularity properties of this map are critical in admitting techniques from abstract boundary control \cite{redbook}, and are determined from the Zaremba elliptic problem \cite{savare}.
  To again decouple the dynamics into a dissipative piece and a perturbation, a trace regularity result must be established. The necessary trace regularity hinges upon the invertibility of an operator which is analogous (in two dimensions) to the finite Hilbert transform \cite{bal0,bal4,tricomi}.   And this is a critical additional element of the challenging  harmonic analysis brought about  by the (KJC). 
When the problem is reduced in dimensionality to a beam structural model, this property can be demonstrated and our analysis  has parallels  with that in \cite{bal0,bal4} (and older references which are featured therein). Specifically, one must invert the finite Hilbert transform in $L_p(\Omega)$ for $p \ne 2$; in higher dimensions, this brings about nontrivial (open) problems in harmonic analysis and the theory of singular integrals when (KJC) is in force.


  The full system with  (\ref{KJC}) can be  written in terms of aeroelastic potential variable as:
\begin{equation}\label{flowplate2}\begin{cases}
u_{tt}+\Delta^2u+f(u)= p_0+tr[\psi] & \text { in } \Om\times (0,T),\\
u=\Dn u=0 & \text{ on } \pd\Om\times (0,T),\\
(\partial_t+U\partial_x)\phi=\psi & \text { in } \realsthree_+ \times (0,T),\\
(\partial_t+U\partial_x)\psi=\Delta \phi & \text { in } \realsthree_+ \times (0,T),\\
\Dn \phi = -(\partial_t+U\partial_x)u & \text{ on } \Omega   \times (0,T),\\
(\partial_t+U\partial_x)  \phi = 0 & \text{ on } \realstwo \backslash \Omega \times (0,T),\end{cases}
\end{equation}
with the given initial data $(u_0,u_1,\phi_0, \psi_0 ) $. 
 We make use of the flow acceleration multiplier $(\partial_t+U\partial_x)\phi \equiv \psi$ for the {\em subsonic flow}, taken with (KJC). 
Thus for the flow dynamics, instead of $(\phi;\phi_t)$ we again have
the state variables  $(\phi;\psi)$ (as in the supersonic panel case)   which then leads to a {\it  non dissipative}  energy balance, see
(\ref{energyrelationkj}).

 Our result is formulated under the following regularity condition (to be discussed later).
 \begin{condition}[{\bf Flow Trace Regularity}]\label{le:FTR0}
  Assume that  $\phi(\xb,t)$  satisfies \eqref{flow} and the Kutta-Joukowsky condition (KJC),
   and $\partial_ttr[\phi],~~ \partial_xtr[\phi]  \in L_2(0,T;H^{-1/2-\epsilon}(\realstwo))~~~~\forall\, T>0.
$
In addition,  $\psi = \phi_t + U \phi_x $ satisfies the estimate
\begin{equation}\label{trace-reg-est-M0}
\int_0^T\|tr[\psi](t)\|^2_{H^{-1/2 -\epsilon} (\realstwo)}dt\le C_T\left(
E_{fl}(0)+
 \int_0^T\| \Dn \phi(t) \|_{\Omega}^2dt\right)
\end{equation}
\end{condition}
We note that the regularity required by Condition \ref{le:FTR0} is exactly the one that will make energy relation meaningful. 
To wit: $u_x(t) \in H^1_0(\Omega)  $ and $ tr[\psi] \in L_2(0,T; H^{- \theta } ), ~\theta > 1/2 $ defines the correct  duality pairing in the tangential direction on $\mathbb R^2$. 
It is also at this point where we use the fact that $u$ satisfies  $ u=0  $ on $\partial \Omega $ . Thus simply supported and clamped boundary conditions imposed on the structure  fully cooperate with this regularity.  
The principal result of \cite{KJ} (see also \cite{fereisel}) reads as follows: 
\begin{theorem}\label{T1}
With reference to the model (\ref{flowplate2}), with  $0 \leq U < 1$:
  Assuming the trace regularity Condition \ref{le:FTR0} holds for the aeroelastic potential $\psi = (\phi_t+U\phi_x)$,
  there exists a unique finite  energy solution which exists on any $[0,T]$.  
  This is to say, for any $T > 0 $, 
  $(\phi, \phi_t, u, u_t ) \in C(0, T; Y )$  for all initial data  $(\phi_0, \phi_1, u_0, u_1 ) \in  Y $. 
  This solution depends continuously on the initial data.  
\end{theorem}
The proof of Theorem \ref{T1} follows the technology developed for the supersonic case in \cite{supersonic} and is given in \cite{KJ}. 
In view of Theorem \ref{T1} the final conclusion on generation of nonlinear semigroup is pending upon verification of the flow regularity Condition. 
While the complete solution  to this question is still unavailable, and pending further progress in the theory of finite Riesz transforms, we can provide positive answer in the case when $\Omega$ is one dimensional. 
In fact, this positive assertion also follows  a-posteriori from direct  calculations given in \cite{bal}, and based on the analysis of finite Hilbert transforms. 
The arguments given below are independent and more general, with potential adaptability   to multidimensional cases. 
The corresponding result is formulated below. 
\begin{theorem}\label{1dtracereg}
Assume $\Omega=(-1,1)$ (suppressing the span variable $y$). Then the flow trace regularity in Condition \ref{le:FTR0} holds for $\phi$.
In this case  the semiflow defined by (\ref{flowplate2}) taken with (KJC) generates a continuous semigroup.
\end{theorem}

In what follows below we shall  address validity of Condition \ref{le:FTR0} in the one dimensional case. 
 In particular, we wish to provide  the relation to $L_p$, ~$p \ne 2 $,  invertibility of Hilbert/Riesz  transforms.\footnote{More  complete analysis and details are given in \cite{KJ}.}
The thrust of the argument is  the trace regularity of the following flow problem in $\realsthree_+$: \begin{equation}\label{flow1}\begin{cases}
(\partial_t+U\partial_x)^2\phi=\Delta \phi & \text { in } \realsthree_+ \times (0,T),\\
\phi(0)=\phi_0;~~\phi_t(0)=\phi_1,\\
\Dn \phi = h(\xb,t)& \text{ on } \Omega \times (0,T)\\
tr[\phi_t+U\phi_x]=tr[\psi]=0, ~~\xb \in \realstwo \backslash \Omega
\end{cases}
\end{equation}
where $h(\xb,t)$, defined on $\Omega$,  is the downwash generated on the structure and $tr [\psi] $ is the aeroelastic potential. The flow trace regularity problem can  be rephrased as the following inverse type boundary-boundary  problem:
\[
\mbox{\bf Given $h(x,y,t)  \in L_2( \mathbb R \times \Omega) $, find $\psi$, and hence $tr[\psi]$, on $\Omega$. }
\]

In line with the similar analyses in \cite{miyatake,sakamoto}, we take zero initial data (later we may apply the principle of superposition). We consider the Fourier-Laplace transform of the original linear equation (formally, sending $(x,y) \to i(\eta_x,\eta_y)$ and $t \to \tau=(\alpha +i\beta$): 
$$ (\tau+iU\eta_x)^2\hat{\phi}+|\eta|^2\hat{\phi}=( \tau^2  + 2U i \eta_x \tau  + (1- U^2) \eta_x^2 + \eta_y^2 ) \hat{\phi} = \Dz^2 \hat{\phi}.$$
Let us introduce the following PDO symbol: 
$$D (\eta, \tau) \equiv (\tau^2  + 2U i \eta_x \tau  + (1- U^2) \eta_x^2 + \eta_y^2), $$ and  
so we must solve (in $z$) the following ODE: 
$$\Dz^2 \hat{\phi} = D (\tau, \eta)  \hat{\phi},$$
whose general solution is 
 $$ \hat{\phi}(\tau,\eta,z) = \hat{\phi} (\tau, \eta, 0) e^{-z \sqrt{D}}. $$ 
 Our next step is to relate the boundary conditions to the the normal derivative of $\phi$  and the acceleration potential $\psi$: 
 $$ \hat{\psi}(\tau,\eta, z) = (\tau + i  U \eta_x )  \hat{\phi} (\tau, \eta, 0) e^{-z \sqrt{D}} $$
 $$ \Dz \hat{\phi} = - \sqrt{D}  \hat{\phi} (\tau, \eta, 0)$$ 
 Hence, on $\partial \realsthree_+$  we have the relation
\begin{equation}\label{ms} \Dz \hat{\phi} = \frac{-\sqrt{D}}{\tau + iU \eta_x } \hat{\psi} = m(\tau, \eta) \hat{\psi} \end{equation}
with
\begin{equation}\label{mul}
m(\eta, \tau)\equiv \frac{-\sqrt{D}}{\tau + iU \eta_x } 
\end{equation}
which corresponds to zero order PDO $\in S^{0,1}$ \cite{hormander,taylor}.
This relation is supplemented with the information 
\begin{eqnarray} 
\Dz \phi|_{z=0} = h \in L_2 (\reals \times \Omega ) \\
{\psi}(t, x, y, 0)  =0 , ~\text{ outside }~\Omega. 
\end{eqnarray}
Let $\cT$ denotes the operator corresponding to the multiplier $m (\eta, \tau)$ in (\ref{ms}). 
Let $P_{\Omega}: \realstwo \to \Omega $ denote projection on $\Omega$. 
Then for any $\psi$ which is zero off  $\Omega$, the problem stated above can be recast as: given $h$ , find $\psi$ so that 
\begin{equation}\label{P}
h =P_{\Omega}\cT\psi, x \in \Omega, t > 0 
\end{equation}
noting that $\cT $ acts on a ``truncated" function $\psi$, so that 
the PDO  
$P_{\Omega} \cT $ can be viewed  as PDO on $\Omega \times (0, \infty) $ (viewing $L_2(\Omega)$ as a subset of $L_2(\realstwo)$ with associated projection). 
We construct $\cT$ as follows:
for any  $\psi$ with $\psi =0 $ outside $\Omega$, we construct Laplace (time) Fourier (space) transform of $\psi$ which
we denote $\hat{\psi} (\tau, \eta) $. Then the operator $\cT $ is the PDO 
such that 
\begin{equation}\label{T}
(\widehat{\cT \psi})(\tau, \eta)  = m(\tau, \eta) \hat{\psi} (\tau, \eta)\end{equation}
 Taking the inverse of the Fourier-Laplace transform, and applying projection on $\Omega$ gives the actual definition of the operator. 
Thus, the problem of trace regularity (PDO  Neumann-Dirichlet map) reduces to solving (\ref{P}) with the PDO   given by (\ref{T})  and  Mikhlin type of multiplier $m(\tau, \eta) $.
 The relationship between $\psi$ and $h$ depends on the properties of the Mikhlin multiplier $m(\tau, \eta )$ in the class of symbols $S^{0,1}$ \cite{hormander,taylor}. In fact, as we will show below, $\cT$ is related to the inverse of a finite Hilbert transform in the $x$ variable. We make this notion precise below, but we do note at this stage that the integral transform which arises is not standard in the two dimensional scenario. It is only when  when $\Omega$ is reduced to an interval that the multiplier above contains the classical {\it finite} Hilbert transform \cite{bal,tricomi}. Otherwise we deal with a ``finite" Riesz transform. 
 If we denote by $\cH$ Hilbert transform with the symbol 
 $symb(\cH) \equiv - i  \frac{|\eta_x|}{\eta_x} $and by $E_{\Omega} $ the extension operator (by zero) from $L_p(\Omega)$  onto $L_p(\mathbb R^n) $, then  {\it finite Hilbert transform} can be written as 
 $$ \cH_f \equiv P_{\Omega} \cH E_{\Omega}  \in \mathscr L(L_p(\Omega) \rightarrow L_p(\Omega) ), 1 < p < \infty.$$ 
 When $\Omega = I $ this corresponds to a classical finite Hilbert transform. For higher dimensions one can consider $\Omega$ as an union of intervals -generalizing the concept to finite Riesz transforms. 
 The main point critical to the analysis is the fact that \cite{FH}
 \begin{equation}\label{Hf}
  \cH_f ^{-1} \in \mathscr L(L_p(\Omega) ), p < 2 
  \end{equation}
 \begin{remark}
 The connection between integral equations appearing in the study of aeroelasticity and invertibility of finite Hilbert fransforms has been known for many years \cite{FH}, and dates back to Tricomi. This approach has been critically used in \cite{bal0,bal4} where the analysis is centered on solvability of integral equations  of Possio's type connecting the downwash with the aeroelastic potential. We follow the same conceptual route as these references, with different technical tools. Our approach is based on microlocal analysis---rather than explicit solvers of integral equations arising in a very special case (of a  purely one dimensional setting). Though our final estimate depends on an assumption which has only been demonstrated for the one dimensional case, we believe that the microlocal approach provides new ground for extending the flow-structure analysis to  multidimensional settings.
 \end{remark}
In line with hyperbolic character of the symbol, it is convenient to introduce the variable 
~$ z_U\equiv \frac{\beta}{\eta} +U.$~
 In the case when the dimension of $\Omega$ is equal to one 
  our multiplier then becomes, 
$$m(z, \eta) = \frac{-i |\eta|}{\eta}  r(z, \eta)  \equiv m_0(\eta) r(z, \eta) $$
 where $$ r(z, \eta) \equiv \frac{\sqrt{D_1}}{ z_U - i \alpha \eta^{-1} },$$
with 
$$ D_1 (z, \eta)  \equiv  -\alpha^2 \eta^{-2}-1 + z_U^2  - 2i \frac{\alpha}{\eta} ( z_U ) .$$
We are interested in the invertibility (and regularity properties of the inverse) of the operator associated to the multiplier $m(z,\eta)$. Noting that $m_0 (\eta) $  is a symbol corresponding to Hilbert transform, we are seeking a bound from below for $r(z, \eta)$. 
The critical regions for the analysis of this symbol occur for large values:  $ |\eta| >> R , |\beta| >> R $, for $R$ sufficiently large.  
Analyzing the symbol $r(z, \eta) $, we note that when ~$ |z| \rightarrow 0 $~ and~ $ |z | \rightarrow \infty $, the symbol $r (z, \eta) $ 
converges to constant, nonzero values. A loss of regularity occurs when $ |z| \rightarrow 1  $. In that case, the asymptotic behavior is characterized by $$r(z, \eta) \sim \frac{c}{\sqrt{\eta}}.$$ 
This indicates a loss of $1/2$ of a spatial derivative, i.e. a bound from below of the form $$r(z,\eta)\sqrt{\eta}>c>0.$$ The above  leads to the following lemma:
\begin{lemma}
For $\Omega=I \subset \reals$ a compact interval, the operator $\cF$ associated with the multiplier $m(\eta)$  is invertible from $L_p(\Omega) \rightarrow W^{-1/2, p} (\Omega) 
$ for $ p \in (1,2) $. \end{lemma}
\begin{proof}
We first reference the fact that finite Hilbert transform is Fredholm  on $L_p(I)$ for $p\in (1,2) \cup (2, \infty) $, in particular (\ref{Hf}) takes place.
\cite{FH1}. 
Moreover, the operator associated with $ m_0(\eta)  \eta^{-1} $ maps $ W^{-1, p }(I)  $ onto $ L_p(I)$, 
for the same range $p \in (1,2)\cup(2,\infty)$. This follows from derivative formula for finite Hilbert transform for functions vanishing outside $I$ \cite{FH}. Thus, by interpolation, the operator associated to the multiplier 
   $m(z, \eta)$  is invertible from $L_p(I)$  into $W^{-1/2, p } (I)$.  Since our analysis is uniform in $\beta$, 
    we have that for $h \in L_2( (0, T) \times I) ) $,
   one obtains that aeroelastic potential  $tr{[\psi]} $  is in $L_2((0, T) , W^{-1/2, p }(I) ) $ for $p < 2 $.
   \end{proof}
   By Sobolev's embeddings $ W^{-1/2, p }(I) ) \subset H^{-1/2 -\epsilon}(I)$ with  $ \epsilon(p) > 0 $ 
   possibly arbitrarily small---justifying the regularity of $tr[\psi]$ in Condition \ref{le:FTR0} when $\Omega = I $. 

\begin{remark}
As discussed above, the generation of semigroups for an arbitrary three dimensional flow is subjected to the validity of the trace Condition \ref{le:FTR0}. While it is believed that this property should be generically true, at the present stage this appears to be an open question in the analysis of singular integrals and depends critically on the geometry of $\Omega$ in two dimensions. 
\end{remark}

\begin{remark}\label{Bal1}
We note that  the  regularity of aeroelastic potential required by  Condition \ref{le:FTR0} in one dimension, follows from the analysis in \cite{bal0,bal4}, where the author proves that 
aeroelastic potential $ \psi \in L_2(0, T , L_q(\Omega) ) $ for $ q < 4/3 $ . Since  for $p > 4 $ there exists $\epsilon > 0 $ 
such that $$H^{1/2 + \epsilon} (\Omega) 
\subset L_p(\Omega) , ~~p > 4 , ~~\text{dim}~~\Omega \leq 2,$$  and one then obtains  that $L_q(\Omega) \subset H^{-1/2 - \epsilon}(\Omega) $ with $q < 4/3 $. \end{remark}
\begin{remark}
The loss of $1/2$ derivative in the characteristic region was already observed and used in the analysis of regularity of 
the aeroelastic potential for the Neumann problem with supersonic velocities $U$ \cite{supersonic}. 
However, in the case of (KJC)  there is an additional loss, due to the necessity of inverting finite Hilbert transform which forces to work with $L_p$ theory for $ p < 2 $. This is due to the fact that  {\it finite} Hilbert transform  is invertible on $L_p$, $p < 2 $ -rather than $ p =2$. 
\end{remark}
 

The author of \cite{bal0,bal4} considers a linear {\em wing}  immersed in a subsonic flow; the wing is taken to have a high aspect ratio thereby allowing for the suppression of the span variable, and reducing the analysis to individual chords normal to the span. By reducing the problem to a one dimensional analysis, many technical hangups are avoided, and Fourier-Laplace analysis is greatly simplified. 

\section{Some Experimental and Computational
Observations Reinforced by  Mathematical Analysis}

In what follows we shall present several observations which, on one hand confirm experimentally and numerically some of the mathematical findings discussed before and, on the other hand, raise open questions and indicate  new avenues for mathematical research. 
\begin{enumerate}
\item{\bf  Stability induced by the flow in reducing dynamics to finite dimensions}:
Experimentally we see that the flow has the ability of inducing stability in the moving structure. This is the case when the structure itself does not possess any mechanical or frictional damping. This dissipative effect is not immediately noticeable in the standard energy balance equation. 
However, rigorous analysis reveals that the stabilizing effect of the flow reduces the structural dynamics to a {\it  finite dimensional } setting. It is remarkable that such conclusion is obtained 
from mathematical considerations---as it is not clearly arrived at by either numerics or experiment. 
\item{\bf Non-uniqueness of final nonlinear state}:
Our analysis of asymptotic behavior asymptotically reduces  the panel dynamics  (in the subsonic case) to the set of stationary solutions.
 However, such stationary solutions are not unique in general.\footnote{
This can be seen from the results in Section 6.3.6 in \cite{springer}
in the subsonic case.}
  In this case it is possible that there exist several \textit{locally} stable equilibria in the global attractor.
This makes it possible to explain the fact---observed experimentally---that the {\em buckled plate} in an aerodynamic flow does not have a final, unique nonlinear state; it can depend on the path which is traveled in the parameter space, and/or initial conditions. See \cite{dowellA} and references cited therein, as well as \cite{dowellB,dowellC,dowellD}. The latter discuss the chaotic responses that may occur for a buckled plate in a gas flow.
\item
{\bf Surprisingly subtle effects of boundary conditions}: 
As already noted, as the structure's boundary conditions are changed, so is the dynamic stability/instability of the system. Experimental and numerical studies reveal that a change of boundary conditions (from clamped to free plates, for instance) can lead from flutter (resulting in periodic solutions), to divergence (resulting in a static buckling). Computationally and experimentally it is found that in subsonic flows the trailing edge boundary condition is the key. If the trailing edge is fixed, then divergence will occur. If it is free, then flutter will occur. A recent computational and experimental study of the effects of boundary conditions on plate flutter, divergence, and limit cycle oscillations has been described in \cite{dowellE}.
Our theoretical results on convergence to equilibria confirm these experimental studies---at least with respect to fixed (clamped) trailing edge. 
\item
{\bf Limitations of VK theory and new nonlinear plate theory}: 
When the leading edge is clamped and the side edges and trailing edge are free flutter and  periodic
oscillations have amplitudes on the order of the plate chord. Thus vK theory is no longer accurate. 
This may explain our mathematical difficulties in carrying the vK  analysis in configurations where a portion of the boundary is free. 
However, a novel, improved nonlinear plate theory has been developed and explored computationally, and correlated with experiment. This provides a novel, challenging, and exciting opportunity for mathematical analysis. A new nonlinear structural theory for cantilevered plates has been developed in \cite{dowellF} and applied to aeroelastic studies including a comparison with experiment in \cite{dowellG}.
\item
{\bf Aerodynamic theory}:
Very recently piston theory has been extended to lower supersonic Mach numbers. This may be helpful to prove mathematically some of the results in the spirit of those given above \cite{newpiston,venedeev2}. Indeed, our analysis of delayed plate indicates that, ultimately, the delayed term (retained for subsonic velocities) does not have any major effect on qualitative behavior. It is rather the dispersive effect (emanating from the ~$-Uu_x$~ term that leads to potential instabilities in the dynamics. 
\item
{\bf Extensions to incorporate solar radiation}:
The pressure due to solar radiation has a similar mathematical form to that of piston theory and has been of recent interest in the context of interplanetary transportation using solar sails. Such sails may also experience flutter due to solar pressure. The nonlinear models for pressure differ, however (aerodynamic versus solar radiation), and represent an opportunity for further mathematical studies. See the work of Dowell in \cite{dowellH} and  Gibbs and Dowell in \cite{dowellI}.
\end{enumerate}

\section{Open Problems}
In what follows below we shall list some open problems which naturally emerge from the material
presented above.
\begin{enumerate}
\item
{\bf Configurations involving a free-clamped plate}: As pointed out in section 1.3, the  {\it free-clamped } boundary conditions imposed on the plate are of great physical interest.
From the mathematical point of view, the difficulty arises at the level of linear theory  when one attempts to construct "smooth" solutions of the corresponding evolution. Typical procedure of extending  plate solutions by zero
outside $\Omega$ leads to the jump in the  Neumann  boundary conditions  imposed on the flow.
In order to contend with this issue, regularization  procedures are needed in order obtain smooth approximations
of the original solutions. While  some regularizations have been already introduced  in \cite{supersonic},
more studies are needed in order to demonstrate the effectiveness  of this "smoothing" for the  large array of problems described in this work.  Free boundary conditions in the context of piston theory and boundary dissipation have been recently studied in \cite{bociu} and also \cite{springer}.
\item {\bf Kutta-Joukowsky boundary conditions for the flow in higher dimensions}: 
As explained above, the (KJC) are used in  the models of panels that are partially free.
Jump of the pressure off the wing is a typical configuration \cite{bal0}. From the mathematical point of view
the difficulty lies, again,  at the level of the  linear theory. In order to deal with   the effects of the {\it unbounded}
traces $tr[\psi] $ in the energy relation microlocal calculus is necessary. This
has been successfully accomplished in \cite{supersonic} where clamped boundary conditions in the supersonic case
were considered. However, in the case of (KJC)  there  is an additional difficulty that involves
``invertibility" of  {\it finite} Hilbert (resp. Riesz) transforms.  This latter property is known to fail within $L_2$ framework, thus it is necessary to build the  $L_p$ theory, $p \ne 2 $.  This was for the first time observed in \cite{bal4}
and successfully resolved in the  one dimensional case. However, any progress to higher dimensions depends
on the validity of the corresponding harmonic analysis result developed for finite Riesz transforms.
Some preliminary results  in the subsonic case are in \cite{KJ}.

\item {\bf The transonic regime}: 
The treatment presented above excludes the  transonic velocity $ U =1$.
Indeed, for $U=1$ the analysis  provided breaks down in the essential way \cite{transnon}.
While numerical, experimental work  predicts appearance of shock waves \cite{C}, to our best knowledge  no mathematical treatment of this problem is  available at present.
 \item {\bf Nonlinear flow models}: 
 Finally, the ultimate goal is to consider a fully   nonlinear flow model. Experimental-numerical results predicting shock waves  in the evolution  are partially
 available \cite{C}. However the mathematical aspects of   this problem are  presently wide open.
 \item {\bf Rigorous comparison (asymptotics) of delay plate dynamics and piston theoretic dynamics}:  There are also a question of justification of the ``piston'' theory at the level of global attractors  via hypersonic limit $U\to+\infty$. This requires asymptotic ($U \rightarrow \infty$)  closeness of the corresponding solution. (See the {\bf Open Question} earlier in the text.)
\end{enumerate}


\section{Acknowledgment}
The authors would like to dedicate this work to Professor A.V. Balakrishnan, whose pioneering and insightful work on flutter brought together engineers and mathematicians alike. 

E.H. Dowell was partially supported by the National Science Foundation with grant NSF-ECCS-1307778.
I. Lasiecka was partially supported by the National Science Foundation with grant NSF-DMS-0606682 and the United States Air Force Office of Scientific Research with grant AFOSR-FA99550-9-1-0459.
 J.T. Webster was partially supported by National Science Foundation with grant NSF-DMS-1504697.

\end{document}